\documentclass[12pt]{article}
\usepackage{ayv_paperstyle}

\title{Zero-sum continuous-time Markov games with one-side stopping}
\author{Yurii Averboukh}
\date{}
\AtEndDocument{
	\vspace{10pt}
	\begin{tabular}{ll}
		Yurii Averboukh:  &
		Higher School of Economics,  \\ & 
		11 Pokrovsky Bulvar,  Moscow, Russia; \\ & 
		 Krasovskii Institute of Mathematics and Mechanics \\ &
		16 S. Kovalevskoi str., Yekaterinburg, Russia; \\ &
		\email{averboukh@gmail.com}
	\end{tabular}
}

\begin{document}
	\maketitle
	
\begin{abstract}
	The paper  is concerned with a variant of the continuous-time finite state Markov game of control and stopping where both players can affect  transition rates, while only one player can choose a stopping time. We use the dynamic programming principle and reduce this problem to a system of ODEs with unilateral constraints. This system plays the role of the Bellman equation. We show that its solution provides the optimal strategies of the players. Additionally, we prove the existence and uniqueness theorem for the deduced system of ODEs with unilateral constraints.
    \msccode{91A05, 91A25, 90C40, 60G40, 60J27}
\end{abstract}
	
\section{Introduction}
The continuous-time Markov processes can be controlled in two ways. First,  the decision maker can affect  transition rates. The second way involves the choice of the stopping time. The case where  the transition rates are controlled refers to the Markov decision problem (see~\cite{Guo_Hernandez_Lerma},~\cite{MDP_survey} and reference therein for the detailed explanation of this field). The optimization of the Markov chain by the means of stopping time is studied within the optimal stopping theory~\cite{Mordecki_stopping, Shiryaev_stopping}. Certainly, one can mix both ways and study Markov decision processes with stopping~\cite{Decision_with_stopping}. The concept of Markov chain affected by several decision makers leads to such settings as the continuous time Markov games those examine the case when all players control transition rates~\cite{Zachrisson1964}, the Dynkin games~\cite{Dynkin1969, Optimal_stopping_zero_sum} where the players choose stopping times and the games of control and stopping where all players can affect the transition rates as well as the stopping time~\cite{Game_with_stopping_control}.     In the paper, we consider one-side stopping Markov game that implies that one player chooses a control and a stopping time, while the other one affects only the control. Our interest to this problem is motivated by the possible application of these games to the differential games with minimal cost studied in~\cite{Barron, Subb_book}. In particular, it is shown  that one can find a one-side stopping Markov game such that its value function approximates the value function of the original differential game with  minimal cost~\cite{Averboukh2021}.

We study the continuous-time finite state Markov chain governed by two decision makers with opposite purposes. Without loss of generality, we assume that the state space is $\mathcal{S}=\{1,\ldots,d\}$. The transition rates are described by the Kolmogorov matrix $Q(t,u,v)=(Q_{i,j}(t,u,v))_{i,j=1}^d$, i.e., for each $t$, $u$ and $v$, \[Q_{i,j}(t,u,v)\geq 0\text{ for }i\neq j,\] while \[\sum_{j\in\mathcal{S}}Q_{i,j}=0.\] Here $u$ (respectively, $v$) is a control of the first (respectively, second) player.  It is assumed that $u\in U$, $v\in V$, while $t\in [0,T]$. We assume that the stopping of the game is chosen by the first player. If $t_0$ is an initial time, $X(t)$ is a state of the Markov chain at time $t$, while $\tau$ is a stopping time, then the quality of the examined one-side Markov game is evaluated by
\begin{equation}\label{payoof:informal:stopping}
	g(\tau,X(\tau))+\int_{t_0}^{\tau}h(t',X(t'),u(t'),v(t'))dt'.
\end{equation} Here $\mathbb{E}$ stands for the expectation. We assume that the first player tries to minimize  functional \eqref{payoof:informal:stopping}, while the second player wishes to maximize it.

The universal tool for the study of  Markov decision processes,  Markov games as well as  optimal stopping problems is the dynamic programming principle. It primary relies on the feedback strategies and  reduces the problem  of design of  optimal strategies to a certain Bellman equation. In the case of the finite horizon Markov decision processes and the Markov games, the Bellman equation takes a form of system of ODEs. For the game theoretical case, this system was derived in  paper by Zachrisson \cite{Zachrisson1964}. Thus, we call it a Zachrisson equation.

In the paper, we use the dynamic programming principle and show the players' optimal strategies can be constructed based on the solution of the Zachrisson equation with unilateral  constraints. Thus, the existence result for the one-side stopping Markov game is reduced to the existence theorem for the Zachrisson equation with unilateral constraints that is also obtained in the paper.

The rest of the paper is organized as follows. In Section~\ref{sect:preliminaries}, we formalize the concept of one-side stopping Markov game and introduce the classes of strategies. The main results are formulated in Section~\ref{sect:main_result}. Here, we, first, claim that, given a solution of the Zachrisson equation with unilateral  constraints, one can design optimal players' strategies. Additionally, in this section, we formulate the existence and uniqueness result for the Zachrisson equation with unilateral constraints. These statements imply the existence of the value function of the examined one-side stopping Markov game. Section~\ref{sect:verification} is concerned with the proof of the first result by the verification argument. The proof of the existence and uniqueness theorem for the Zachrisson equation with unilateral constraints is given in Section~\ref{sect:existence}.

\section{Preliminaries}\label{sect:preliminaries}
In the paper, we identify every function $\phi:\mathcal{S}\rightarrow \mathbb{R}$ with the corresponding vector in $\mathbb{R}^{d}$ and assume that $\phi(i)$ and $\phi_i$ have the same meaning.

Given an initial time $t_0$, we fix the measurable space $(\Omega_{t_0},\mathcal{F}_{t_0})$ letting $\Omega_{t_0}$ to be equal to the space of c\`{a}dl\`{a}g functions from $[t_0,T]$ to $\mathcal{S}$ and choosing $\mathcal{F}_{t_0}$ to be a Borel $\sigma$-algebra on $\Omega_{t_0}$, i.e., $\Omega_{t_0}\triangleq D([t_0,T],\mathcal{S})$, and $\mathcal{F}_{t_0}\triangleq \mathcal{B}(D([t_0,T],\mathcal{S}))$. 

If $s,r\in [t_0,T]$, $s<r$, then denote by $\mathcal{R}^{s,r}$ operator assigning to $\omega\in \Omega_{t_0}=D([t_0,T],\mathcal{S})$ its restriction on $[s,r]$. 

Further, we put $\mathcal{F}_{t_0,t}$ to be the $\sigma$-algebra of sets $A\in\Omega_{t_0}$ such that $\mathcal{R}^{t_0,t}(A)\in\mathcal{B}(D([t_0,t],\mathcal{S}))$ and $\mathcal{R}^{t,T}(A)\in \{\varnothing,D([t,T],\mathcal{S})\}$. Obviously, $\mathbb{F}_{t_0}\triangleq \{\mathcal{F}_{t_0,t}\}_{t\in [t_0,T]}$ is a filtration on $\mathcal{F}_{t_0}$. We denote the set of stopping time w.r.t. the filtration $\mathbb{F}_{t_0}$ taking values in $[t_0,T]$ by $\mathcal{T}_{t_0}$.

If $\phi:[0,T]\times\mathcal{S}\rightarrow \mathbb{R}$ is differentiable at $s\in [0,T]$ for some $i\in \mathcal{S}$, and $\omega\in \Omega_{t_0}$ is such that $\omega(s)=i$, we use the notation 
\[\frac{\partial}{\partial t}\phi(s,\omega(s))\] for the derivative of the function $\phi$ w.r.t. the time variable at the point $s$, $i=\omega(s)$. 

Below, we denote by $\mathcal{AC}$ the set of all functions $\phi:[0,T]\times \mathcal{S}\rightarrow \mathbb{R}$ such that mapping $t\mapsto\phi(t,i)$ is absolutely continuous for each $i\in\mathcal{S}$.

We consider two types of strategies. First, we will assume that the player has information only about current state and time. This leads to the concept of feedback strategies. 

\begin{definition}\label{def:strategy_I} A   measurable w.r.t. the time  mapping $u$ from $[0,T]\times\mathcal{S}$ to $U$ is called a first player's feedback strategy.
\end{definition} We denote the set of all feedback strategies of the first player by $\mathcal{U}_{\operatorname{fb}}$. As it was mentioned above the first player also chooses the stopping time. Within the feedback approach, it is controlled by the stopping rule defined as follows.
\begin{definition}\label{def:stopping_rule}
	A c\'adl\'ag function  from $[0,T]\times\mathcal{S}$ to $\{0,1\}$ is called a stopping rule. 
\end{definition} We will assume that the game stops at the first time when the stopping rule is equal to $1$, i.e., if $\vartheta$ is a stopping rule, then the corresponding stopping time on $[t_0,T]$ is defined by the rule:
\[\bar{\tau}_{t_0}[\vartheta](\omega)\triangleq \min\{t\in [t_0,T]:\vartheta(t,\omega(t))=1\}.\]
The set all stopping rules is denoted by $\mathcal{T}_{\operatorname{fb}}$.  

The feedback strategies of the second player are defined in the same way.
\begin{definition}\label{def:strategy_II}
	We say that a mapping $v:[0,T]\times\mathcal{S}\rightarrow V$ is a second player's strategy provided that $v$ is measurable w.r.t. the time variable.
\end{definition} The set of feedback strategies is denoted by $\mathcal{V}_{\operatorname{fb}}$.

Since we consider the guaranteed results of the players, it is reasonable to follow the approach of~\cite{NN_PDG} and enlarge of the set of strategies admitting strategies with memory. 

\begin{definition}\label{def:progressively_strategies} If $t_0$ is an initial time, a $\mathbb{F}_{t_0}$-progressively measurable process with values in $U$ (respectively, $V$) is called a strategy with memory of the first (respectively, second) player on $[t_0,T]$.  
\end{definition} We denote the set of strategies with memory of the first (respectively, second) player on $[t_0,T]$ by $\mathcal{U}_{t_0}$ (respectively, $\mathcal{V}_{t_0}$).

The link between feedback strategies and strategies with memory is straightforward. If $u\in\mathcal{U}_{\operatorname{fb}}$, then $\alpha[u]$ defined by the rule \[\alpha[u](t,\omega)\triangleq u(t,\omega(t))\] is a strategy with memory. Analogously, given $v\in\mathcal{V}_{\operatorname{fb}}$, we define the corresponding strategy with memory $\beta[v]$ setting $\beta[v](t,\omega)\triangleq v(t,\omega(t))$.

The probability on $\mathcal{F}_{t_0}$ generated by two strategies, initial time and initial state is defined in the following way.
\begin{definition}\label{def:probability}
Given two strategies with memory $\alpha\in\mathcal{U}_{t_0}$ and $\beta\in\mathcal{V}_{t_0}$, an initial time $t_0\in [0,T]$ and an initial state $i_0\in\mathcal{S}$, we say that the probability $\mathbb{P}^{\alpha,\beta}_{t_0,i_0}$ on $\mathcal{F}_{t_0}$  is generated by $\alpha$, $\beta$, $t_0$ and $i_0$ provided that
\begin{itemize}
	\item $\mathbb{P}_{t_0,i_0}^{\alpha,\beta}(\{\omega:\omega(t_0)=i_0\})=1$.
\item
 for every $\phi\in\mathcal{AC}$, the process
\begin{equation}\label{proc:phi}
	\phi(t,\omega(t))-\int_{t_0}^t\bigg[\frac{\partial}{\partial t}\phi(t',\omega(t'))+\sum_{j\in\mathcal{S}}Q_{\omega(t'),j}(t',\alpha(t',\omega),\beta(t',\omega))\phi(t',j)\bigg]dt'
\end{equation} is $(\mathbb{P}_{t_0,i_0}^{\alpha,\beta},\mathbb{F})$-martingale.
\end{itemize}
\end{definition} Below $\mathbb{E}_{t_0,i_0}^{\alpha,\beta}$ stands for the expectation corresponding to the probability $\mathbb{P}_{t_0,i_0}^{\alpha,\beta}$.

 If $\alpha=\alpha[u]$ for some feedback strategy of the first player $u$, we will write 
$\mathbb{P}_{t_0,i_0}^{u,\beta}$ instead of $\mathbb{P}_{t_0,i_0}^{\alpha[u],\beta}$. Analogously, $\mathbb{P}_{t_0,i_0}^{\alpha,v}\triangleq\mathbb{P}_{t_0,i_0}^{\alpha,\beta[v]}$. Notice that the probability $\mathbb{P}^{u,v}_{t_0,i_0}$ is always well-defined. Additionally, one can prove that $\mathbb{P}_{t_0,i_0}^{\alpha,\beta}$ exists in the case when both $\alpha$ and $\beta$ depend on a finite number of time instants.

The quality of the players' strategies is evaluated as follows. Given initial time $t_0\in [0,T]$, initial state $i_0\in\mathcal{S}$, a first player's strategy $\alpha\in\mathcal{U}_{t_0}$, a stopping time $\tau\in \mathcal{T}_{t_0}$ and a strategy of the second player $\beta\in\mathcal{V}_{t_0}$, the corresponding payoff within the one-side stopping Markov game is equal to
\[\mathbb{E}_{t_0,i_0}^{\alpha,\beta}G_{t_0}^{\alpha,\beta}(\tau),\] where we denote
\[G_{t_0}^{\alpha,\beta}(t,\omega)\triangleq g(t,\omega(t))+\int_{t_0}^t h(t',\omega(t'),\alpha(t',\omega),\beta(t',\omega))dt'.\] As above, in the case when $\alpha=\alpha[u]$ (respectively, $\beta=\beta[v]$), we will write $G_{t_0}^{u,\beta}$  (respectively, $G_{t_0}^{\alpha,v}$).

\section{Main result}\label{sect:main_result}

We impose the following assumptions:
\begin{enumerate}[label=(H\arabic*)]
	\item\label{assumption:compact} The control spaces $U$ and $V$ are compact;
	\item\label{assumption:continuous} the transition rates $Q_{i,j}$ and the payoff functions $g(\cdot,i)$, $h(\cdot,i,\cdot,\cdot)$ are continuous for each $i$ and $j$;
	\item\label{assumption:g_Lipschit} the function $g$ is Lipschitz continuous w.r.t. the time variable;
	\item\label{assumption:Isaacs} for every $t\in [0,T]$, $i\in\mathcal{S}$ and each $w\in\rd$,
	\[\begin{split}\min_{u\in U}\max_{v\in V}\sum_{j\in\mathcal{S}}&{}\Bigl[Q_{i,j}(t,u,v)w_j+h(t,i,u,v)\Bigr]\\&=\max_{v\in V}\min_{u\in U}\sum_{j\in\mathcal{S}}\Bigl[Q_{i,j}(t,u,v)w_j+h(t,i,u,v)\Bigr].\end{split}\]
\end{enumerate} Notice that assumption~\ref{assumption:Isaacs} can be considered as an analog of the Isaacs' condition in the theory of differential games \cite{NN_PDG}.

We  consider the upper and lower values of the game those corresponds to the upper and lower games. We assume that within the upper game the first player chooses his/her strategy (that includes both controls and stopping time) from the class of feedback strategies, while the response of the second player is from the class of strategies with memory. Formally, this  means the advantage of the second player. In the lower game, we give the advantage to the first player. The second player's choice is restricted to the class of feedback strategies, while the first player chooses his/her policy from the class of strategies with memory. 

Therefore, the upper and lower values of the one-side stopping Markov game are defined by the rules:
\[\operatorname{Val}^+(t_0,i_0)\triangleq \inf_{u\in \mathcal{U}_{\operatorname{fb}}}\inf_{\vartheta\in\mathcal{T}_{\operatorname{fb}}}\sup_{\beta\in \mathcal{V}_{t_0}}\mathbb{E}_{t_0,i_0}^{u,\beta}G_{t_0}^{u,\beta}(\bar{\tau}_{t_0}[\vartheta]),\]
\[\operatorname{Val}^-(t_0,i_0)\triangleq \sup_{v\in \mathcal{V}_{\operatorname{fb}}}\inf_{\alpha\in \mathcal{U}_{t_0}}\inf_{\tau\in\mathcal{T}_{t_0}}\mathbb{E}_{t_0,i_0}^{\alpha,v}G_{t_0}^{\alpha,v}(\tau).\]

Notice that due to the fact that, for each $t_0\in [0,T]$, $i_0\in\mathcal{S}$, $u\in\mathcal{U}_{\operatorname{fb}}$, $v\in \mathcal{V}_{\operatorname{fb}}$, the probability $\mathbb{P}^{u,v}_{t_0,i_0}$ is well-defined, we have that all aforementioned values are finite. 

Moreover, 
\begin{equation}\label{ineq:Values_upper_lower_st}
	\operatorname{Val}^+(t_0,i_0)\geq \operatorname{Val}^-(t_0,i_0).
\end{equation} 

Below, we show that the one-side stopping Markov games has a value. Additionally, we will show that this value can be obtained as a solution of the Zachrisson  equation with unilateral constraints. The latter object is defined as follows. For $w\in\rd$, $i\in \mathcal{S}$,  set
\begin{equation}\label{intro:Hamiltonian}H_i(t,w)\triangleq \min_{u\in U}\max_{v\in V}\sum_{j\in\mathcal{S}}\Bigl[Q_{i,j}(t,u,v)w_j+h(t,i,u,v)\Bigr].\end{equation} Notice that, due to assumption~\ref{assumption:Isaacs}, 
\[
	H_i(t,w)= \max_{v\in V}\min_{u\in U}\sum_{j\in\mathcal{S}}\Bigl[Q_{i,j}(t,u,v)w_j+h(t,i,u,v)\Bigr].\] Below, if $\phi:[0,T]\times\mathcal{S}\rightarrow\mathbb{R}$, the $\phi(t)$ stands for the vector $(\phi(t,1),\ldots,\phi(t,d))^T$.

\begin{definition}\label{def:Zachrisson} We say that a  function $\varphi\in\mathcal{AC}$ satisfies the Zachrisson equation with unilateral  constraints if, for each $i\in\mathcal{S}$,
	\begin{enumerate}[label=(Z\arabic*)]
		\item\label{ZE:boundary} $\varphi(T,i)=g(T,i)$;
		\item\label{ZE:upper_bound} $\varphi(t,i)\leq g(t,i)$ for $t\in [0,T]$;
		\item\label{ZE:u_stability} for every point $t\in [0,T]$ such that $\frac{d}{dt}\varphi(t,i)$ is well-defined and $\varphi(t,i)<g(t,i)$, one has
		\[\frac{d}{dt}\phi(t,i)\leq -H_i(t,\varphi(t));\]
		\item\label{ZE:v_stability} if  $\phi(\cdot,i)$ is differentiable at $t\in [0,T]$, then
		\[\frac{d}{dt}\phi(t,i)\geq -H_i(t,\varphi(t)).\]
	\end{enumerate}
\end{definition}
\begin{theorem}\label{th:verification}
	Let $\varphi$ satisfy the Zachrisson equation with unilateral constraints. Then at every $t_0\in [0,T]$, $i_0\in\mathcal{S}$,
	\[\operatorname{Val}^+(t_0,i_0)= \operatorname{Val}^-(t_0,i_0)=\varphi(t_0,i_0).\]
\end{theorem}

\begin{theorem}\label{th:existence} There exists a unique function $\varphi$ satisfying the Zachrisson equation with unilateral constraints.
\end{theorem}

These two theorems obviously yield the following.
\begin{corollary}\label{corr:value} There exists the value of the one-side stopping Markov game.
\end{corollary}

Theorem~\ref{th:verification} and~\ref{th:existence} are proved in Sections~\ref{sect:verification} and~\ref{sect:existence} respectively.

\section{Suboptimal strategies}\label{sect:verification}
To prove Theorem~\ref{th:verification}, we consider  suboptimal strategies of the players. Let us start with the construction of suboptimal strategy of the first player in the one-side stopping Markov game. 

Given a function $\phi\in\mathcal{AC}$, we define a strategy $\hat{u}_\phi\in\mathcal{U}_{\operatorname{fb}}$ and a stopping rule $\hat{\vartheta}_{\phi,t_0}\in\mathcal{T}_{\operatorname{fb}}$ by the rules:
	\begin{equation}\label{intro:u_suboptima}
	\hat{u}_\phi(t,i)\in \underset{u\in U}{\operatorname{Argmin}}\Bigl[\sum_{j\in\mathcal{S}}Q_{i,j}(t,u,v)\varphi(t,j)+h(t,i,u,v)\Bigr],
\end{equation}
\[\hat{\vartheta}_{\phi}(t,i)=\left\{\begin{array}{cc}
1, & \phi(t,i)=g(t,i), \\
0, & \phi(t,i)\neq g(t,i).
\end{array}\right.\] The latter corresponds to the following stopping time with values on the time interval $[t_0,T]$
\begin{equation}\label{intro:tau_suboptimal}
	\hat{\tau}_{\phi,t_0}(\omega)=\bar{\tau}_{t_0}[\hat{\vartheta}_{\phi}](\omega)=\min\{t\in [t_0,T]:\phi(t,\omega(t))=g(t,\omega(t))\}.
\end{equation}

\begin{proposition}\label{prop:u_stable} Assume that $\varphi$ satisfies conditions~\ref{ZE:boundary},~\ref{ZE:u_stability} of Definition~\ref{def:Zachrisson}. Then, for every $t_0\in [0,T]$, $i_0\in\mathcal{S}$ and each $\beta\in \mathcal{U}_{t_0}$, the following inequality holds true:
\begin{equation}\label{prop_ustable:ineq:value}
	\mathbb{E}^{\hat{u}_\varphi,\beta}_{t_0,i_0}G_{t_0}^{\hat{u}_\varphi,\beta}(\hat{\tau}_{\varphi,t_0})\leq \varphi(t_0,i_0)
\end{equation}
\end{proposition}
\begin{proof}
	First, due to assumption~\ref{ZE:boundary}, we have that the stopping time $\hat{\tau}_{\varphi,t_0}$ takes values in~$[t_0,T]$. 
	
	By the optional sampling theorem  \cite[Theorem 7.29]{Kallenberg} and the facts that process ~\eqref{proc:phi} with $\alpha=\alpha[\hat{u}_\varphi]$  is a $(\mathbb{P}^{\hat{u}_{\varphi},\beta}_{t_0,i_0},\mathbb{F}_{t_0})$-martingale, while $\mathbb{P}^{\hat{u}_{\varphi},\beta}_{t_0,i_0}(\{\omega:\omega(t_0)=i_0\})=1$, we have that 
	\[\begin{split}\varphi(t_0,i_0)=\mathbb{E}_{t_0,i_0}^{\hat{u}_\varphi,\beta}\bigg[ \phi&(\hat{\tau}_{\varphi,t_0},\omega(\hat{\tau}_{\varphi,t_0}))-\int_{t_0}^{\hat{\tau}_{\varphi,t_0}}\Bigl[\frac{\partial}{\partial t}\phi(t',\omega(t'))\\&+\sum_{j\in\mathcal{S}}Q_{\omega(t'),j}(t',\hat{u}_\varphi(t',\omega(t')),\beta(t',\omega))\phi(t',j)\Bigr]dt'\bigg].\end{split}\] The assumption that $\varphi$ satisfies condition~\ref{ZE:u_stability}  implies that 
	\[\frac{\partial}{\partial t}\varphi(t',\omega(t'))\leq -H(t',\varphi(t',\omega(t')))\] for a.e. $t'\in [t_0,\hat{\tau}_{\varphi,t_0}(\omega)]$. This and the definition of the stopping rule $\hat{\vartheta}_{\varphi}$ (that, in its own turn, determines the stopping time $\hat{\tau}_{\phi,t_0}$ see~\eqref{intro:tau_suboptimal}) give that
	\[\begin{split}\varphi(t_0,i_0)\geq \mathbb{E}_{t_0,i_0}^{\hat{u}_\varphi,\beta}g(\hat{\tau}_{\varphi,t_0},\omega(\hat{\tau}_{\varphi,t_0}&))\\+\mathbb{E}_{t_0,i_0}^{\hat{u}_\varphi,\beta}\int_{t_0}^{\hat{\tau}_{\varphi,t_0}}&{}\Bigl[H_i(t',\varphi(t'))\\&-\sum_{j\in\mathcal{S}}Q_{\omega(t'),j}(t',\hat{u}(t',\omega(t')),\beta(t',\omega))\phi(t',j)\Bigr]dt'.\end{split}\] Due to the definitions of $H_i$ (see~\eqref{intro:Hamiltonian}) and the strategy $\hat{u}_\varphi$ (see~\eqref{intro:u_suboptima}), we have that
	\[\begin{split}
	H_i(t',\varphi(t'))-\sum_{j\in\mathcal{S}}Q_{\omega(t'),j}(t',\hat{u}_\varphi(t',&\omega(t')),\beta(t',\omega))\phi(t',j)\\&\geq h(t',i,\hat{u}_\varphi(t',\omega(t')),\beta(t',\omega)).\end{split}\] Therefore, 
	\[\varphi(t_0,i_0)\geq \mathbb{E}_{t_0,i_0}^{\hat{u}_\varphi,\beta}\Bigl[g(\hat{\tau}_{\varphi,t_0},\omega(\hat{\tau}_{\varphi,t_0}))+\int_{t_0}^{\hat{\tau}_{\varphi,t_0}}h(t',i,\hat{u}_\varphi(t',\omega(t')),\beta(t',\omega))dt'\Bigr].\]  Using the definition of the process $G$, we obtain inequality~\eqref{prop_ustable:ineq:value}.
\end{proof}

Now we consider the second player's suboptimal strategy $\tilde{v}$ that is defined on the base of a function $\phi\in\mathcal{AC}$ as follows:
\begin{equation}\label{intro:v_suboptimal}
	\tilde{v}_\phi(t,i)\in \underset{v\in V}{\operatorname{Argmax}}\min_{u\in U}\Bigl[\sum_{j\in\mathcal{S}}Q_{i,j}(t,u,v)\varphi(t,j)+h(t,i,u,v)\Bigr].
\end{equation}

\begin{proposition}\label{prop:v_suboptimal}
	Let $\varphi\in\mathcal{AC}$ satisfy condition~\ref{ZE:upper_bound},~\ref{ZE:v_stability}. Then, for every $t_0\in [0,T]$, $i_0\in\mathcal{S}$, $\alpha\in \mathcal{U}_{t_0}$, $\tau\in\mathcal{T}_{t_0}$,
	\begin{equation}\label{proc:v_suboptimal:ineq:expect}
		\mathbb{E}_{t_0,i_0}^{\alpha,\tilde{v}_{\varphi}}G_{t_0}^{\alpha,\tilde{v}_\varphi}(\tau)\geq \varphi(t_0,i_0).
	\end{equation}
\end{proposition}
\begin{proof}
	The proof is also by the verification arguments. Using the fact that  process ~\eqref{proc:phi} with $\phi=\varphi$ and $\beta=\beta[\tilde{v}_{\varphi}]$ is a $(\mathbb{P}_{t_0,i_0}^{\alpha,\tilde{v}_{\varphi}},\mathbb{F}_{t_0})$-martingale, the equality $\mathbb{P}_{t_0,i_0}^{\alpha,\tilde{v}_{\varphi}}(\{\omega:\omega(t_0)=i_0\})=1$ and optional sampling theorem \cite[Theorem 7.29]{Kallenberg}, we conclude that
	\[\begin{split}\varphi(t_0,i_0)=\mathbb{E}_{t_0,i_0}^{\alpha,\tilde{v}_{\varphi}}\bigg[\varphi(&\tau,\omega(\tau)) \\&-\int_{t_0}^\tau\Bigl[\frac{\partial}{\partial t}\varphi(t',\omega(t'))+\sum_{j\in\mathcal{S}}Q_{\omega(t'),j}(t,\alpha(t',\omega),\tilde{v}(t',\omega(t')))dt' \Bigr]\bigg].
	\end{split}\] Applying conditions~\ref{ZE:upper_bound},~\ref{ZE:v_stability}, the definitions of the Hamiltonian (see~\eqref{intro:Hamiltonian}) and the strategy $\tilde{v}_\varphi$ (see~\eqref{intro:v_suboptimal}), we obtain the following estimate:
	\[\varphi(t_0,i_0)\leq\mathbb{E}_{t_0,i_0}^{\alpha,\tilde{v}_{\varphi}}\bigg[g(\tau,\omega(\tau)) -\int_{t_0}^\tau h(t',\omega(t'),\alpha(t',\omega),\tilde{v}_\varphi(t',\omega(t')))\bigg].\]
	This gives the statement of the proposition.
\end{proof}

\begin{proof}[Proof of Theorem~\ref{th:verification}]
This theorem directly follows from Propositions~\ref{prop:u_stable},~\ref{prop:v_suboptimal} and inequality~\eqref{ineq:Values_upper_lower_st}.
\end{proof}

\section{Existence of solution for the Zachrisson equation with unilateral constraints}\label{sect:existence}
To prove the existence of the solution of the Zachrisson equation with unilateral constraints, we consider the sequence of functions $\varphi_N:[0,T]\times\mathcal{S}\rightarrow \mathbb{R}$ defined as follows.

First, we set $t^k_N\triangleq Tk/N$. Then, let 
\begin{equation}\label{exist:itro:phi_N_T}
	\varphi_N(T,i)\triangleq g(T,i).
\end{equation} Now, assume that the function $\varphi_N$ is already defined for all $t\in [t^{k+1}_N,T]$ and every $i\in \mathcal{S}$. Let $\psi_N^k:[t^k_N,t^{k+1}_N]\times \mathcal{S}\rightarrow \mathbb{R}$ solve the boundary value problem on $[t^k_N,t^{k+1}_N]$:
\begin{equation}\label{exist:eq:psi_N}
	\frac{d}{dt}\psi_N^k(t,i)=-H_i(t,\psi_N^k(t)),\ \ \psi_N^k(t^{k+1}_N,i)=\varphi_N(t^{k+1}_N,i), \ \ i\in\mathcal{S}.
\end{equation}
Finally, set, for $t\in [t^k_N,t^{k+1}_N]$,
\begin{equation}\label{exist:intro:phi_N_k}
	\varphi_N(t,i)\triangleq \psi_N^k(t,i)\wedge g(t,i).
\end{equation}

Now examine the properties of the sequence $\{\varphi_N\}_{N=1}^\infty$.

\begin{lemma}\label{lm:phi_N_boundness} There exists a constant $C_1$ such that, for every natural $N$, $t\in [0,T]$, $i\in\mathcal{S}$,
	\[|\varphi_N(t,i)|\leq C_1.\]
\end{lemma}
\begin{proof}
	First, denote
	\begin{equation}\label{exist:intro:M_Q}
		M_Q\triangleq \max\{|Q_{i,j}(t,u,v)|:i,j\in\mathcal{S},\, t\in [0,T],\, u\in U,\, v\in V\},
	\end{equation}
	\begin{equation}\label{exist:intro:M_h}
	M_h\triangleq \max\{|h(t,i,u,v)|: t\in [0,T],\, i\in\mathcal{S},\, u\in U,\, v\in V\},
\end{equation}
	\begin{equation}\label{exist:intro:M_g}
	M_g\triangleq \max\{|g(t,i)|:t\in [0,T],\, i\in\mathcal{S}\}.
\end{equation} Due to  assumptions~\ref{assumption:compact},~\ref{assumption:continuous}, the quantities $M_Q$, $M_h$ and $M_g$ are finite.

Put
\begin{equation}\label{exist:intro:M_function}
	\mathcal{M}(t)\triangleq (M_g+M_h(T-t))e^{dM_Q(T-t)}.
\end{equation}

We will prove the estimate
\begin{equation}\label{exist:ineq:phi_N_M}
	|\varphi_N(t,i)|\leq \mathcal{M}(t)
\end{equation} by the backward induction. 

By the definition of the function $\mathcal{M}$,~\eqref{exist:ineq:phi_N_M} holds true at $t=T$. 

Further, assume that~\eqref{exist:ineq:phi_N_M} is fulfilled on $[t^{k+1}_N,T]$. Let us show that the functions 
\begin{equation}\label{exist:intro:xi_N_k}\xi_N^k(t)\triangleq \max_{i\in\mathcal{S}}|\psi^k_N(t,i)|\end{equation} satisfies on $[t^k_N,t^{k+1}_N]$ the following inequality:
\begin{equation}\label{exist:ineq:psi_N_absolute}
	\xi_N^k(t)\leq \mathcal{M}(t).
\end{equation} Indeed, by~\eqref{exist:eq:psi_N} and~\eqref{intro:Hamiltonian} we have that 
\[\xi^k_N(t)\leq \mathcal{M}(t^{k+1}_N)+\int_{[t,t^{k+1}_N]}dM_Q\xi^k_N(t')dt'+M_h(t^{k+1}_N-t).\] Applying the Gronwall's inequality and definition of the function $\mathcal{M}$ (see~\eqref{exist:intro:M_function}), we conclude that~\eqref{exist:ineq:psi_N_absolute} holds true on $[t^k_N,t^{k+1}_N]$. The definition of $\varphi_N$ on $[t^k_N,t^{k+1}_N]$ (see~\eqref{exist:intro:phi_N_k}) and the fact that $g$ is bounded by $M_g$, while $\mathcal{M}(t)\geq M_g$ yield estimate~\eqref{exist:ineq:phi_N_M} on $[t^k_N,t^{k+1}_N]$.

To complete the proof, let
\[C_1\triangleq \mathcal{M}(0).\]

\end{proof}

\begin{lemma}\label{lm:continuous_psi} If $k\in \{0,\ldots,N-1\}$, $s,r\in [t^k_N,t^{k+1}_N]$, then
		\begin{equation*}|\psi^k_N(s,i)-\psi^k_N(r,i)|\leq C_2|s-r|,\end{equation*} where $C_2$ is a constant depending only on $g$, $h$ and $Q$,
\end{lemma}
\begin{proof}
	Since $\psi_N^k$ satisfies ODE \eqref{exist:eq:psi_N}, it suffices to estimate the right-hand side of this equation. First, due to the definition of the Hamiltonian and the constants $M_Q$, $M_h$ (see~\eqref{intro:Hamiltonian},~\eqref{exist:intro:M_Q},~\eqref{exist:intro:M_h}), we have that 
	\[|H_i(t,\psi_N^k)|\leq dM_Q\max_{i\in\mathcal{S}}|\psi^k_N(t,i)|+M_h.\] Using the estimate of $\max_{i\in\mathcal{S}}|\psi^k_N(t,i)|$  by the function $\mathcal{M}$ (see \eqref{exist:intro:xi_N_k}, \eqref{exist:ineq:psi_N_absolute}) and the fact that $\mathcal{M}(t)\leq C_1$, we obtain the statement of the lemma with  $C_2\triangleq dM_QC_1+M_h$.
\end{proof}

\begin{lemma}\label{lm:equicontinuous} There exists a constant $C_3$ such that, for every natural $N$, $s,r\in [0,T]$, $i\in\mathcal{S}$,
	\begin{equation}\label{lm_equicont:ineq:phi}|\varphi_N(s,i)-\varphi_N(r,i)|\leq C_3|s-r|.\end{equation}
\end{lemma}
\begin{proof}

	Recall (see assumption~\ref{assumption:g_Lipschit}) the  function $g$ is Lipschitz continuous w.r.t. the time variable. Let $L_g$ denote this Lipschitz constant. Put
	\begin{equation}\label{exist:intro:C_3}
		C_3\triangleq L_g\vee C_2.
	\end{equation} 

Further, let us consider  time instants $t',t''$ such that they lies in the same time interval $[t^k_N,t^{k+1}_N]$ for some $k$. Without loss of generality, one can assume that $\varphi_N(t',i)>\varphi_N(t'',i)$. To estimate $|\varphi_N(t',i)-\varphi_N(t'',i)|=\varphi_N(t',i)-\varphi_N(t'',i)$, we consider two cases.
\begin{enumerate}
	\item $\varphi_N(t'',i)< g(t'',i)$. Therefore, $\varphi_N(t'',i)=\psi_N^k(t'',i)$, while $\varphi_N(t',i)\leq \psi_N^k(t',i)$. The latter inequality directly follows from the definition of the function $\varphi_N$ (see~\eqref{exist:intro:phi_N_k}). By Lemma~\ref{lm:continuous_psi}, we have that
	\[\varphi_N(t',i)-\varphi_N(t'',i)\leq \psi^k_N(t',i)-\psi_N(t'',i)\leq C_2|t'-t''|\leq C_3|t'-t''|.\]
	\item $\varphi_N(t'',i)= g(t'',i)$. As above, by definition of the function $\varphi_N$ (see \eqref{exist:intro:phi_N_k}), $\varphi_N(t',i)\leq g(t',i)$. Hence,
	\[\varphi_N(t',i)-\varphi_N(t'',i)\leq g(t',i)-g(t'',i)\leq L_g|t'-t''|\leq C_3|t'-t''|.\]
	
\end{enumerate} Therefore, we derive~\eqref{lm_equicont:ineq:phi} for $s,r$ lying in some interval $[t^k_N,t^{k+1}_N]$. To complete the proof it suffices to sum up inequalities~\eqref{lm_equicont:ineq:phi} on small time intervals.
\end{proof}

\begin{lemma}\label{lm:u_stability_N} Let $t\in [0,T]$ and $i\in\mathcal{S}$ be such that $\phi_N(t,i)<g(t,i)$ and $\varphi_N$ is differentiable at $t$. Then 
	\begin{equation}\label{lm_u_stab:ineq:derivative}\frac{d}{dt}\phi_N(t,i)\leq -H_i(t,\varphi_N(t)).\end{equation}
\end{lemma}
\begin{proof}  First, let us consider the case where $t\in (t^k_N,t^{k+1}_N)$ for some $k$. Due to the assumption that $\phi_N(t,i)<g(t,i)$, we have that \[\phi_N(t,i)=\psi_N^k(t,i).\] Moreover, the inequality $\phi_N(t',i)<g(t',i)$ holds in some neighborhood of $t$. Thus, 
	\[\frac{d}{dt}\phi_N(t,i)=\frac{d}{dt}\psi_N^k(t).\] By construction of the function $\psi_N^k$, we have that
	\begin{equation*}\label{exist:equality:derivative_inner}\frac{d}{dt}\varphi_N(t,i)=\frac{d}{dt}\psi_N^k(t)=-H_i(t,\psi_N^k(t)).\end{equation*}
	
	Further, notice that $\varphi_N(t,j)\leq \psi_N^k(t,j)$, $Q_{i,j}(t,u,v)\geq 0$  for $i\neq j$. This and the equality $\phi_N(t,i)=\psi_N^k(t,i)$ imply that, for every $u\in U$ and $v\in V$,
	\begin{equation}\label{exist:ineq:Ham_phi_psi}\sum_{j\in\mathcal{S}}Q_{i,j}(t,u,v)\varphi_N(t,j)+h_i(t,i,u,v)\leq \sum_{j\in\mathcal{S}}Q_{i,j}(t,u,v)\psi_N^k(t,j)+h_i(t,i,u,v).\end{equation} Therefore, using the definition of the Hamiltonian (see~\eqref{intro:Hamiltonian}), we obtain the inequality
	\[H_i(t,\varphi_N(t))\leq H_i(t,\psi_N^k(t)).\] Combining this with~\eqref{exist:equality:derivative_inner}, we conclude that~\eqref{lm_u_stab:ineq:derivative} holds true for $t\neq t^k_N$,

	Now let us consider the case when $t=t^{k+1}_N$ for some $k$. By construction of the function $\varphi_N$, we have that $\varphi_N(t_N^{k+1})=\psi_N^{k}(t^{k+1}_N)$. Additionally, due to the assumption $\varphi_N(t_N^{k+1},i)<g(t_N^{k+1},i)$ and the continuity of these functions, we conclude that 
	 $\varphi_N(t',i)<g(t',i)$ when $t'\in (t_N^{k+1}-\delta,t_N^{k+1}]$. Hence, the construction of the function $\varphi_N$ gives that $\varphi_N(t',i)=\psi_N^{k}(t',i)$ on $(t_N^{k+1}-\delta,t_N^{k+1}]$. Therefore,
	 \[\frac{d}{dt}\varphi_N(t^{k+1}_N,i)=\frac{d^-}{dt^-}\psi_N^{k}(t_N^{k+1},i)=-H_i(t_N^{k+1},\psi_N^{k-1}(t_N^{k+1})).\] Here $\frac{d^-}{dt-}$ stands for the left-hand derivative. This and~\eqref{exist:ineq:Ham_phi_psi} yield that \eqref{lm_u_stab:ineq:derivative} holds true in  the case when $t\in \{t^k_N\}_{k=1}^{N-1}$.
\end{proof}

\begin{lemma}\label{lm:v_stability_N} There exists a constant $C_4$ such that, for every $s,r\in [0,T]$, $s<r$, and $i\in \mathcal{S}$, the following inequality holds true:
	\[\varphi_N(r,i)-\varphi_N(s,i)\geq \int_s^r H_i(t,\varphi_N(t))dt-\frac{C_4}{N}\]
\end{lemma}
\begin{proof}
	Let $m,n\in \{0,\ldots,N-1\}$ be such that $s\in [t^m_N,t^{m+1}_N)$, $r\in (t^n_N,t^{n+1}_N]$. 
	
	We have that 
	\[\begin{split}
	\varphi_N(r,i)-\varphi_N(s,i)=\varphi_N(r,i)-\varphi_N(&t^n_N) +\varphi_N(t^{m+1}_N,i)-\varphi_N(s,i) \\&+\sum_{k=m+1}^{n-1}(\varphi_N(t^{k+1}_N,i)-\varphi_N(t^k_N,i))\\ \geq\varphi_N(r,i)-\varphi_N(&t^n_N) +\varphi_N(t^{m+1}_N,i)-\varphi_N(s,i) \\&+\sum_{k=m+1}^{n-1}(\psi_N^k(t^{k+1}_N,i)-\psi^k_N(t^k_N,i))
	.\end{split} \] In the latter inequality, we used the definition of the functions $\psi_N^k$ and $\varphi_N$ (see~\eqref{exist:eq:psi_N} and~\eqref{exist:intro:phi_N_k}). Further, due to Lemma~\ref{lm:equicontinuous},
$|\varphi_N(r,i)-\varphi_N(t^n_N)|, \, |\varphi_N(t^{m+1}_N,i)-\varphi_N(s,i)|\leq C_3T/N$. Using this and the fact that each function $\psi^k_N(\cdot)$ satisfies~\eqref{exist:eq:psi_N}, we conclude that, for each $k$, 
\begin{equation}\label{exist:ineq:varphi_s_r}
\varphi_N(r,i)-\varphi_N(s,i)\geq -2C_3\frac{T}{N}+\sum_{k=m+1}^{n-1}\int^{
t^{k+1}_N}_{t^k_N}H_i(t,\psi_N^k(t))dt.\end{equation} By the equality $\psi_N^k(t_N^{k+1})=\varphi_N(t_N^{k+1})$,  we have that
\[\begin{split}H_i(t,\psi_N^k(t))=H_i(t,\varphi_N(t))+(H_i&(t,\psi_N^k(t))-H_i(t,\psi_N^k(t_N^{k+1})))\\&+(H_i(t,\varphi_N(t_N^{k+1}))-H_i(t,\varphi_N(t))).\end{split}\]
Lemmas~\ref{lm:continuous_psi},~\ref{lm:equicontinuous} and the fact that $H_i$ is Lipschitz continuous w.r.t. the second variable with the constant $dM_Q$ give that, for $t\in [t_N^k,t^{k+1}_N]$,
\begin{equation}\label{exist:ineq:H_i} H_i(t,\psi_N^k(t))\geq H_i(t,\varphi_N^k(t))-(C_3+C_2)dM_Q\frac{T}{N}.\end{equation} Finally, using the definition of Hamiltonian (see~\eqref{intro:Hamiltonian}) and Lemma~\ref{lm:phi_N_boundness}, we conclude that
\[|H_i(t,\varphi_N(t))|\leq C'.\] Here $C'$ is a constant dependent only on the examined game. This, inequalities~\eqref{exist:ineq:varphi_s_r} and~\eqref{exist:ineq:H_i} and the fact that  $r-t_N^n,t^{m+1}_N-s\leq T/N$ imply that 
\[\begin{split}	\varphi_N(r,i)-\varphi_N(s,i)\geq -(3C_3+&C_2+C')\frac{T}{N}+\int_{s}^{t^{m+1}_N}H_i(t,\varphi_N(t))dt\\&+\sum_{k=m+1}^{n-1}\int^{
	t^{k+1}_N}_{t^k_N}H_i(t,\varphi_N(t))dt+\int^{r}_{t^{n}_N}H_i(t,\varphi_N(t))dt\\=-(3C_3+&C_2+C')\frac{T}{N}+\int_{s}^{r}H_i(t,\varphi_N(t))dt.
\end{split} \] This gives the statement of the lemma.

\end{proof}

\begin{proof}[Proof of Theorem~\ref{th:existence}]
	
As it was mentioned above, to prove the existence of the Zachrisson equation with  unilateral constraints, we use the sequence of functions $\{\varphi_N\}_{N=1}^\infty$. Due to Lemma~\ref{lm:phi_N_boundness}, these functions are uniformly bounded. Lemma~\ref{lm:equicontinuous} says that these functions are Lipschitz continuous with the  same constant. By the Arzela–Ascoli theorem, there	exist a sequence $\{\varphi_{N_l}\}_{l=1}^\infty$ and a function $\varphi$ such that $\varphi_{N_l}\rightarrow \varphi$ as $l\rightarrow\infty$. The limiting function $\varphi$ is bounded by the constant $C_1$ and is Lipschitz continuous with the constant $C_3$. Hence,
\[\varphi\in\mathcal{AC}.\]

Further, by construction, each function $\varphi_{N_l}$ satisfies the boundary condition
\[\varphi_{N_l}(T)=g(T).\] Passing to the limit, we obtain that $\varphi$ satisfies condition~\ref{ZE:boundary}. Analogously,~\eqref{exist:intro:phi_N_k} implies that 
\[\varphi_{N_l}(t,i)\leq g(t,i),\ \ i\in\mathcal{S}.\] Therefore, condition~\ref{ZE:upper_bound} is fulfilled for the limiting function $\varphi$.

Let $i\in\mathcal{S}$ and let $t\in (0,T)$ be a point of differentiability of the functions $\varphi(\cdot,i)$ such that $\varphi(t,i)<g(t,i)$. Choose $\varepsilon$ satisfying
\[3\varepsilon\leq g(t,i)-\varphi(t,i).\] Due to the continuity of the functions $g$ and $\varphi$, we have that there exists $\delta>0$ such that, for $t'\in (t-\delta,t+\delta)$,
\[\varphi(t',i)+2\varepsilon<g(t',i).\] Using the fact that $\varphi_{N_l}\rightarrow\varphi$ as $l\rightarrow\infty$, we have that, for sufficiently large $l$ and every $t'\in (t-\delta,t+\delta)$,
\[\varphi_{N_l}(t',i)+\varepsilon<g(t',i).\]  Lemma~\ref{lm:u_stability_N} implies that, if $\hat{t}\in (t-\delta,t+\delta)$,
\[\varphi_{N_l}(\hat{t},i)-\varphi_{N_l}(t)\leq -\int_t^{\hat{t}}H_i(t',\varphi_{N_l}(t'))dt'.\] Passing to the limit, we conclude that
\[\varphi(\hat{t},i)-\varphi(t)\leq -\int_t^{\hat{t}}H_i(t',\varphi(t'))dt'\] when $\hat{t}\in (t-\delta,t+\delta)$. Since $t\in (0,T)$ is a point of differentiability of the function $\varphi(\cdot,i)$, we arrive at condition~\ref{ZE:u_stability}.

To prove that $\varphi$ satisfies~\ref{ZE:v_stability}, notice that Lemma~\ref{lm:v_stability_N} implies that, for every $s,r\in [0,T]$, $s<r$,
\[\varphi_{N_l}(r)-\varphi_{N_l}(s)\geq -\int_s^r H_i(t,\varphi_{N_l}(t))dt-C_4/N. \] Passing to the limit, we have
\[\varphi(r)-\varphi(s)\geq -\int_s^r H_i(t,\varphi(t))dt. \] This gives condition~\ref{ZE:v_stability} at every point of differentiability of the function $\varphi$.

The uniqueness  follows from Theorem~\ref{th:verification}. Indeed, each functions satisfying~\eqref{def:Zachrisson} coincides with both upper and lower values.

\end{proof}

\begin{acknowledgement}
	The paper was prepared within the framework of the HSE University Basic Research Program in 2022.
\end{acknowledgement}

\bibliography{markov_stopping}

\begin{thebibliography}{10}

\bibitem{Averboukh2021}
Y.~V. Averboukh.
\newblock Approximation of value function of differential game with minimal
  cost.
\newblock {\em Vestn. Udmurt. Univ.: Mat. Mekhanika Komp'yuternye Nauki},
  31(4):536--561, 2021.

\bibitem{Barron}
E.~N. Barron.
\newblock Differential games with maximum cost.
\newblock {\em Nonlinear Anal. Theory Methods Appl.}, 14(11):971--989, 1990.

\bibitem{Dynkin1969}
E.~B. Dynkin.
\newblock Game variant of a problem on optimal stopping.
\newblock {\em Soviet Math. Dokl.}, 10:270--274, 1969.

\bibitem{Optimal_stopping_zero_sum}
E.~B. Frid.
\newblock The optimal stopping rule for a two-person markov chain with opposing
  interests.
\newblock {\em Theory Probab. Appl.}, 14(4):713--716, 1969.

\bibitem{Game_with_stopping_control}
M.~K. Ghosh and K.~S.~M. Rao.
\newblock Zero-sum stochastic games with stopping and control.
\newblock {\em Oper. Res. Lett.}, 35(6):799--804, 2007.

\bibitem{MDP_survey}
X.~Guo, O.~Hern{\'a}ndez-Lerma, and T.~Prieto-Rumeau.
\newblock A survey of recent results on continuous-time {Markov} decision
  processes (with comments and rejoinder).
\newblock {\em Top}, 14(2):177--261, 2006.

\bibitem{Guo_Hernandez_Lerma}
X.~P. Guo and O.~Hern\'{a}ndez-Lerma.
\newblock {\em Continuous-time Markov Decision Processes: Theory and
  Applications}.
\newblock Springer, New York, 2009.

\bibitem{Decision_with_stopping}
M.~Horiguchi.
\newblock Markov decision processes with a stopping time constraint.
\newblock {\em Math. Methods Oper. Res.}, 53(2):279--295, 2001.

\bibitem{Kallenberg}
O.~Kallenberg.
\newblock {\em Foundations of Modern Probability}.
\newblock Springer, New York, 2021.

\bibitem{NN_PDG}
N.~N. Krasovskii and A.~I. Subbotin.
\newblock {\em Game-theoretical control problems}.
\newblock Springer, New York, 1988.

\bibitem{Mordecki_stopping}
E.~Mordecki.
\newblock Optimal stopping and perpetual options for {L}\'{e}vy processes.
\newblock {\em Finance Stoch.}, 6:473--493, 2002.

\bibitem{Shiryaev_stopping}
A.~N. Shiryaev.
\newblock {\em Optimal Stopping Rules}.
\newblock Springer, New York, 2008.

\bibitem{Subb_book}
A.~I. Subbotin.
\newblock {\em Generalized solutions of first-order PDEs. The dynamical
  perspective}.
\newblock Birkhauser, Boston, 1995.

\bibitem{Zachrisson1964}
L.~Zachrisson.
\newblock Markov games.
\newblock {\em Ann. Math. Stud.}, 52:211--253, 1964.

\end{thebibliography}
\end{document}